\documentclass[12pt,reqno,oneside]{amsart}

\usepackage{amssymb}
\usepackage{amsmath}
\usepackage[latin1]{inputenc}

\usepackage[all]{xy}
\usepackage{pstricks, pst-3d}
\newcommand{\mm}{\mathfrak m}

\newcommand{\qq}{\mathfrak q}

\newcommand{\Z}{\mathbb{Z}}
\newcommand{\R}{\mathbb{R}}
\newcommand{\N}{\mathbb{N}}

\newcommand{\Fc}{\mathcal{F}}

\newcommand{\Mcc}{\mathcal{M}}

\DeclareMathOperator{\pnt}{\raise 0.5mm \hbox{\large\bf.}}

\DeclareMathOperator{\Spec}{Spec}

\DeclareMathOperator{\depth}{depth}
\DeclareMathOperator{\relint}{int}

\DeclareMathOperator{\chara}{char}

\DeclareMathOperator{\Hom}{Hom}

\DeclareMathOperator{\st}{star}

\def\+#1{\relax\ifmmode\if\noexpand #1\relax \mathop{\kern
   0pt^+{#1}}\nolimits\else \kern 0pt^+\!#1 \fi\else$^*$#1\fi}

\let\phi=\varphi
\let\:=\colon

\newtheorem{thm}{\bf Theorem}[section]
\newtheorem{lem}[thm]{\bf Lemma}
\newtheorem{cor}[thm]{\bf Corollary}
\newtheorem{prop}[thm]{\bf Proposition}

\newtheorem{quest}[thm]{\bf Question}

\theoremstyle{definition}
\newtheorem{defn}[thm]{\bf Definition}

\newtheorem{rem}[thm]{\bf Remark}
\newtheorem{ex}[thm]{\bf Example}

\theoremstyle{plain}
\newtheorem*{thm*}{Theorem}
\newtheorem*{quest*}{Question}

\title{Seminormality and local cohomology of toric face rings}
\author{Dang Hop Nguyen}
\address{Fachbereich Mathematik/Informatik, Institut f\"ur Mathematik, Universit\"at Osnabr\"uck, Albrechtstr. 28a, 49069 Osnabr\"uck, Germany}
\email{nhop@uos.de}

\thanks{The author has been supported by the graduate school ``Combinatorial Structures in Algebra and Topology'' at the University of Osnabr\"uck, Germany.}

\begin{document}

\begin{abstract}
We characterize the toric face rings that are normal (respectively seminormal). Extending results about local cohomology of Brun, Bruns, Ichim, Li
and R\"omer of seminormal monoid rings and Stanley toric face rings, we prove the vanishing of certain graded parts of local cohomology of
seminormal toric face rings. The combinatorial formula we obtain generalizes Hochster's formula. We also characterize all (necessarily seminormal)
 toric face rings that are $F$-pure or $F$-split over a field of characteristic $p>0$. An example is given to show that $F$-injectivity does not
 behave well with respect to face projections of toric face rings. Finally, it is shown that weakly $F$-regular toric face rings are normal affine
monoid rings.
\end{abstract}

\maketitle
\bigskip

\section{Introduction}
\label{intro}

Combinatorial commutative algebra utilizes techniques and constructions from combinatorics to study problems in commutative algebra related to monomial subrings or monomial quotients of polynomial rings. One of the classical applications of combinatorial commutative algebra is the solution of Stanley \cite{Sta1} to the Upper Bound Conjecture proposed by Klee on triangulations of the sphere.

The class of toric face rings was first defined by Stanley in \cite{Sta2}. One of the main advantages of studying toric face rings is that we can often unify separated results in two seemingly unrelated classes of rings, namely Stanley-Reisner rings and affine monoid rings. Later this class was generalized by allowing more flexible monoidal structures and studied by M. Brun, W. Bruns, B. Ichim, R. Koch, T. R\"omer in \cite{BR}, \cite{BR2}, \cite{BKR} and R. Okazaki, K. Yanagawa in \cite{OkaYan}. The author considered the problem of characterizing Koszul toric face rings in \cite{Nguyen}.

Our basic objects of study in this paper are the following. We have a field $k$, a rational pointed fan in $\mathbb{R}^d$ denoted by $\Sigma$ (where $d\ge 1$) and a monoidal complex $\Mcc$ supported on $\Sigma$. Let $R=k[\Mcc]$ be the toric face ring of $\Mcc$. Naturally, $R$ is $\mathbb{Z}^d$-graded. We call toric face rings with these conditions {\em embedded} toric face rings. Denote by $\mm$ the unique graded maximal ideal of $R$.

B. Ichim and T. R\"omer \cite{IR} studied embedded toric face rings. They were able to extend various results from the theory of Stanley-Reisner rings to toric face rings. For example, if the underlying monoidal complex is {\em pure shellable} and the monoid rings $k[M_C]$ are Cohen-Macaulay for all $C\in \Sigma$, then the toric face ring $k[\Mcc]$ is also Cohen-Macaulay. This generalizes the result saying that a pure shellable simplicial complex is Cohen-Macaulay. Ichim and R\"omer also determined the canonical modules and Gorenstein criterion for {\em Stanley toric face rings}, i.e.~ toric face rings with $M_C=C\cap \Z^d$ for every $C\in \Sigma$. The method developed in \cite{IR} also gives a compact derivation of the formula of graded local cohomology of toric face rings in Brun, Bruns and R\"omer \cite{BBR}.

R. Okazaki and K. Yanagawa \cite{OkaYan} determined the dualizing complex of toric face rings under the normality assumption of the monoid rings involved. On the other hand, Okazaki and Yanagawa did not assume that the toric face rings considered are endowed with a $\mathbb{Z}^d$-grading. They also relaxed the condition ``$M_C=C\cap \Z^d$ for every $C\in \Sigma$''. Thus they provided more general results than those of Ichim and R\"omer \cite{IR}. They also defined squarefree modules over a toric face ring and characterize the Cohen-Macaulay, Buchsbaum and Gorenstein${}^*$ properties of toric face rings with the normality assumption.

In an interesting paper, Bruns, Li and R\"omer \cite{BLR} considered seminormal affine monoid rings, generalized on the way the known results for the classical case of normal affine monoid rings. For example, they give certain results extending Hochster's theorem on Cohen-Macaulayness of normal affine monoid rings. One of the main purposes of this paper is to extend further the results of \cite{BLR} and \cite{IR} to seminormal toric face rings and toric face rings with a $\mathbb{Z}^d$-grading in general. Somewhat surprisingly, we find that the method of studying Stanley-Reisner rings can give new results in the case of seminormal toric face rings and thus, new results in the case of seminormal affine monoid rings, see Section \ref{depth}.

The paper is organized as follows. In Section \ref{background}, we recall the basic notions and results concerning convex geometry
of cones and polytopes, affine monoids and monoidal complexes, toric face rings and seminormality of monoids. In Section \ref{semi},
 we prove that normal toric face rings are precisely normal affine monoid rings. We prove that $R$ is seminormal if and only if
for all $C\in \Sigma$, the monoid $M_C$ is seminormal. We explain the construction of the seminormalization of $\Mcc$ following
Bruns and Gubeladze \cite[Exercise 8.13]{BG} at the end of Section \ref{semi}.

In the Section \ref{lc}, we extend the results of Bruns, Li and R\"omer concerning local cohomology of seminormal monoid rings. We prove the following vanishing result for seminormal toric face rings.
\begin{thm*}
Let $\Sigma$ be a rational pointed fan in $\mathbb{R}^d$ (where $d\ge 1$), $\Mcc$ a seminormal monoidal complex supported
on $\Sigma$, and $R=k[\Mcc]$. Assume that $H^i_{\mm}(R)_a \neq 0$ for some $a\in \mathbb{Z}^d$.
Then $a\in - \overline{M_C}$ for a cone $C\in \Sigma$ of dimension $\le i$. In particular,
\[
H^i_{\mm}(R)_a = 0 ~ \textnormal{if} ~ a\notin -|\overline{\mathcal M}|=\cup _{C\in \Sigma}(-\overline{M_C}).
\]
\end{thm*}
After proving a combinatorial formula computing local cohomology of seminormal toric face rings, we ask whether seminormalization
can be characterized by vanishing of certain graded parts of local cohomology.
\newpage
\begin{quest*}
Is it true that for a toric face ring $R=k[\Mcc]$ the following statements are equivalent?
\begin{enumerate}
 \item $\mathcal M$ is seminormal;
\item $H^i_{\mm}(R)_a=0$ for all $i$ and all  $a\in \mathbb{Z}^d$ such that $a\notin -|\overline{\mathcal M}|$.
\end{enumerate}
\end{quest*}
Note that this is indeed the case for affine monoid rings, as proved in \cite[Theorem 4.7]{BLR}.

Hochster described graded local cohomology of Stanley-Reisner rings in terms of combinatorial data of the associated simplicial complexes. In the second part of Section \ref{lc} we extend Brun, Bruns and R\"omer's generalized version of Hochster formula for seminormal toric face rings via Theorem \ref{lc-combin}. We deduce the formula of Brun, Bruns and R\"omer at the end of this section.

In Section \ref{depth}, we apply the combinatorial description of local cohomology (Theorem \ref{lc-combin}) to study the depth of seminormal toric face rings. We are able to prove a rank-selection theorem describing depth in terms of the skeletons of the monoidal complex. This is well-known for simplicial complexes but is new for affine monoids to our knowledge.

In the last section, given a toric face ring $k[\Mcc]$ over a field $k$ with $\chara k=p>0$, we consider the problem of characterizing the $F$-purity and $F$-splitness of $k[\Mcc]$. Using a result of Schwede \cite{Sch}, we prove that firstly $\Mcc$ is seminormal as long as $k[\Mcc]$ is $F$-injective. Then we describe all possible values of $p$ depending on $\Mcc$ such that $k[\Mcc]$ is $F$-pure. These are also the values of $p$ such that $k[\Mcc]$ is $F$-split. We give an example showing that in general, $F$-injectivity is not stable under projections of a toric face ring onto its faces. In the end, we prove that weak $F$-regularity is a strong condition, so that a toric face rings is weakly $F$-regular if and only if it is a normal affine monoid ring.

The main content of this paper is part of the author's PhD dissertation \cite{Nguyenthesis}.
\section{Preliminaries}
\label{background}

In this paper, we denote by $\R, \R_+, \Z$ the sets of real numbers, non-negative real numbers and integer numbers, respectively.

We consider $\R^d$ with fixed coordinates (where $d\ge 1$). A {\em linear hyperplane} in $\R^d$ is a hyperplane which contains
the origin. A linear hyperplane in $\R^d$ defines two closed {\em linear half spaces} of $\R^d$. A {\em cone} $C$ in $\R^d$ is
 a finite intersection of linear half spaces.

We say that $C$ is {\em rational}, if each of the hyperplanes which cut out $C$ is defined by a homogeneous linear
equation with integral coefficients.

A {\em supporting hyperplane} $H$ of $C$ is a linear hyperplane for which $C$ is contained in one of the two half spaces defined
by $H$ and $H\cap C\neq \emptyset$. The non-empty intersection of $C$ with a supporting hyperplane is called a {\em face} of $C$.
Faces of a cone are again cones. The set of all faces of $C$, including $C$ itself forms a finite partially ordered set
(by inclusion) called the {\em face poset} of $C$, since the relation of being a face of a cone is transitive. We denote the
face poset of a cone $C$ by $\Fc(C)$.

$C$ is called {\em pointed} if $\{0\}$ belongs to the face poset $\Fc(C)$. We assume that the cones considered in this paper are rational and pointed.

An {\em affine monoid} $M$ is a finitely generated submonoid of $\Z^d$ for some $d\ge 1$. Denote by $\Z M$ the subgroup $\{x-y:x,y\in M\}$ of
$\Z^d$. The set of finite non-negative real combinations of elements from $M$ is denoted by $\R_+M$. In fact, $\R_+M$ is a cone in $\R^d$.

We say that $M$ is {\em positive} if $z,-z\in M$ implies that $z=0$. We have $\R_+M$ is a pointed cone if and only if $M$ is a positive monoid.

The {\em normalization} of $M$ in $\Z M$ is the monoid
\[
\overline{M}=\{x\in \Z M: px\in M ~ \text {for some $p\in \Z, ~ p>0$} \}.
\]
Gordan's lemma says that $\overline{M}$ is also an affine monoid. $M$ is a {\em normal monoid} if $M=\overline{M}$.
\begin{defn}
A set $\Sigma$ of cones in $\R^d$ is called a {\em fan} if the two conditions:
\begin{enumerate}
\item if $C\in \Sigma$ and $D$ is a face of $C$ then $D\in \Sigma$;
\item if $C,C'\in \Sigma$ then $C\cap C'$ is a common face of $C$ and $C'$,
\end{enumerate}
are satisfied.
\end{defn}
If all the cones in $\Sigma$ are rational and pointed then we call $\Sigma$ a {\em rational pointed fan}.
Let $\Sigma$ be a rational pointed fan in $\R^d$. An {\em embedded monoidal complex} $\Mcc$ supported on $\Sigma$
is a collection of affine monoids $M_C$ parameterized by $C\in \Sigma$ satisfying the following conditions:
\begin{enumerate}
\item $M_C\subseteq C\cap \Z^d$ and $\R_+M_C=C$ for each $C\in \Sigma$;
\item if $D$ is a face of $C\in \Sigma$ then $M_D= M_C\cap D$.
\end{enumerate}
Denote by $|\Mcc|$ the set $\cup_{C\in\Sigma}M_C$, called the support of $\Mcc$. Let $\{a_1,\ldots,a_n\}$ be a set of generators
for $|\Mcc|$, i.e., $\{a_1,\ldots,a_n\}\cap M_C$ generates $M_C$ for each $C\in\Sigma$. Denote $S=k[X_1,\ldots,X_n]$.

Assume that $C_1,\ldots,C_r$ are the maximal cones of $\Sigma$, and denote $M_i=M_{C_i}, i=1,\ldots,r$. Let $I_1,\ldots,I_r$ be
the binomial defining ideal of the affine monoid rings $k[M_1], \ldots, k[M_r]$, respectively, thus $I_t$ is the kernel of the epimorphism
\[
k[X_i:i\in\{1,2,\ldots,n\},a_i\in M_t]\twoheadrightarrow k[M_t].
\]
We will define the {\em toric face ideal} as
\[
I_{\Mcc}=A_{\Mcc}+ \sum_{i=1}^r S\cdot I_i,
\]
where $A_{\Mcc}$ is the ideal generated by all the squarefree monomials $X_{i_1}\cdots X_{i_j}$ for which $\{a_{i_1},\ldots,a_
{i_j}\}\nsubseteq C_i$ for all $i=1,\ldots,r$.

Then the {\em toric face ring} of $\Mcc$ over $k$ is $k[\Mcc]=S/I_{\Mcc}$. Note that $k[\Mcc]$ does not depend on the generators
$a_1,\ldots,a_n$. We also have that as a $k$-vector space
\[
k[\Mcc] = \bigoplus_{a\in |\Mcc|}k\cdot X^a,
\]
and the multiplication on $k[\Mcc]$ is given by the addition on the monoids $M_C$,
\[
X^a\cdot X^b= \begin{cases}
              X^{a+b} & \text{if for some $C\in \Sigma$ both $a$ and $b$ belong to $M_C$};\\
                       0 & \text{otherwise.}
              \end{cases}
\]
For $a\in |\Mcc|$, sometimes we write $a$ instead of $X^a$, and $ab$ instead of $X^a\cdot X^b$. It would be clear from the
context which meaning should be attributed to the corresponding notation.
In the case $M_C=C\cap \Z^d$ for all $C\in \Sigma$, we call the resulting ring $k[\Sigma]$ a {\em Stanley toric face ring}.
\begin{ex}
(i) If all the cones $C\in \Sigma$ are generated by linearly independent vectors and $M_C$ is generated by $\dim C$ elements, then for $i=1,\ldots,r$, the ideal $I_i=0$ and $k[\Mcc]$ is a {\em Stanley-Reisner ring}.

(ii) On the other hand, if $r=1$ or in other words, $\Sigma$ is the face poset of a cone, then $k[\Mcc]$ is the {\em affine monoid ring}
$k[M_1]$.
\end{ex}
Thus one can say that toric face rings are a natural generalization of Stanley-Reisner rings and affine monoid rings. The reader might wish
to consult Bruns-Herzog \cite{BH}, Chapter 5 and 6 for a detailed discussion of these two kind of rings.

\begin{ex}[\cite{Nguyen}, Example 4.6]
Consider the points in $\R^3$ with the following coordinates $O =(0,0,0), A_1 = (2,0,0),
A_2=(0,2,0), A_3 = (0,0,2), A_4 = (1,1,0).$

Consider the rational pointed fan $\Sigma$ in $\R^3$ with the maximal cones $\normalfont {OA_1A_2}$, $\normalfont {OA_1A_3}$, and
$\normalfont{OA_2A_3}$. In other words, $\Sigma$ is the boundary complex of the cone generated by the three positive axes of $\R^3$. Let $\Mcc$ be
 the monoidal complex supported on $\Sigma$ with the three maximal monoids generated by $\{A_1, A_2, A_4\}$, $\{A_1,A_3\}$ and $\{A_2,A_3\}$.

The defining ideal of the corresponding toric face ring is
$$
I_{\Mcc}= (X_3X_4, X_1X_2X_3)+ (X_1X_2-X_4^2)=(X_1X_2-X_4^2, X_3X_4).
$$
Thus the toric face ring of $\Mcc$ is
$$
k[\Mcc]=k[X_1,X_2,X_3,X_4]/(X_1X_2-X_4^2, X_3X_4).
$$

\bigskip
\bigskip

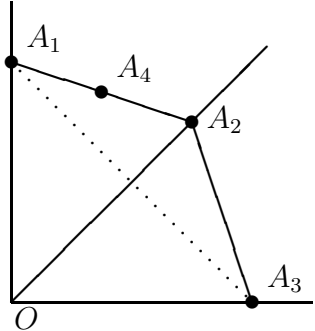
\setlength{\unitlength}{4cm}
\begin{picture}(1,1)
\thicklines
\put(1,0){\line(0,1){1}}
\put(1,0){\line(1,0){1}}
\put(1,0){\line(1,1){0.85}}

\put(1.01,-0.09){$O$}
\put(1.05,0.85){$A_1$}
\put(1,0.8){\circle*{0.04}}
\put(1.85,0.05){$A_3$}
\put(1.8,0){\circle*{0.04}}
\put(1.65,0.58){$A_2$}
\put(1.6,0.6){\circle*{0.04}}
\put(1.35,0.75){$A_4$}
\put(1.3,0.7){\circle*{0.04}}

\put(1,0.8){\line(3,-1){.6}}
\put(1.8,0){\line(-1,3){.2}}
\multiput(1,0.8)(0.03,-0.03){27}{\circle*{0.01}}
\end{picture}
\bigskip
\bigskip

\end{ex}

Next, we recall the projection on to the faces of a toric face ring. We have a natural inclusion $k[M_C]\hookrightarrow k[\Mcc]$ for
 every $C\in \Sigma$. Moreover, we have the {\em face projection} morphism $k[\Mcc]\twoheadrightarrow k[M_C]$ which is given by
$$
X^a\mapsto \begin{cases}
              X^a & \text{if $a\in M_C$};\\
                0 & \text{otherwise.}
              \end{cases}
$$
Note that the composition of $k[M_C]\hookrightarrow k[\Mcc]$ and $k[\Mcc]\twoheadrightarrow k[M_C]$ is the identity of $k[M_C]$.
Thus $k[M_C]\hookrightarrow k[\Mcc]$ is an {\em algebra retract} for each $C\in \Sigma$.

For each $C, D\in \Sigma$ with $D$ is a face of $C$, we also have a natural projection map $k[M_C] \twoheadrightarrow k[M_D]$
defined in the same way.

\begin{prop}[\cite{BKR}, Proposition 2.2]
\label{inverselimit}
We always have
$k[\Mcc] \cong \varprojlim  k[M_C]$.
\end{prop}

For more discussions of basic ring-theoretic properties of toric face rings, we refer the reader to \cite{IR}.

\begin{defn}[Swan \cite{Swa}]
\label{seminormalitydefn}
A reduced ring $R$ is called {\em seminormal} if whenever $x,y\in R$ are such that $x^2=y^3$, then we find a $z\in R$ such that
$x=z^3, y=z^2$.
\end{defn}
Equivalently, a reduced ring $R$ is seminormal if and only if for every $z$ in the total ring of fractions $Q(R)$ of $R$ such that
$z^2, z^3\in R$, we have $z\in R$. The following result follows easily from Definition \ref{seminormalitydefn}.

\begin{prop}[\cite{Swa}, Corollary 3.3]
\label{inverselimitofseminormality}
If $R= \varprojlim R_{\alpha}$ and all $R_{\alpha}$ are seminormal, then $R$ is seminormal.
\end{prop}

An affine monoid $M$ is called {\em seminormal} if for every $x\in \Z M$ such that $2x, 3x\in M$, we have $x\in M$.
Note that a normal monoid is always seminormal. Denote $\+M$ the intersection of all seminormal submonoid of $\Z M$ which contains
$M$. We call $\+M$ the {\em seminormalization} of $M$. Then $\+M$ is contained in the normalization $\overline{M}$ of $M$ in
$\Z M$ and $\+M$ is again an affine monoid.

Hochster and Roberts \cite[Proposition 5.32]{HR} proved that an affine monoid ring $k[M]$, which is always domain, is a
seminormal ring if and only if $M$ is a seminormal monoid. L. Reid and L. Roberts proved the following formula for the
seminormalization of a monoid.
\begin{thm}[\cite{RR}, Theorem 4.3]
Let $M\in \Z^d$ be an affine monoid. Then
\[
\+M= \bigcup_{F ~ \textnormal{is face of} ~ \R_+M}\Z (M\cap F) \cap \relint F,
\]
where $\relint F$ denote the relative interior of the cone $F$.
\end{thm}

In particular, if $M$ is seminormal then we have $\Z M\cap \relint \R_+M \subseteq M$.

\section{Normality and seminormality of toric face rings}
\label{semi}

We will show that a normal toric face ring is nothing but a normal affine monoid ring.
First we prove a general statement about reduced graded algebra over a field.

\begin{lem}
\label{reducedgraded}
Let $R$ be a reduced $\mathbb{N}^d$-graded affine $k$-algebra which is generated by elements of non-zero degrees ($d\ge 1$). If $R$ is normal then $R$ is a domain.
 \end{lem}
\begin{proof}
Assume that $R$ is not a domain. As $R$ is normal, it must be a non-trivial finite product of normal domains. In that case, $\Spec R$ is disconnected.
However, this is not the case.

Indeed, the minimal prime ideals of $R$ are $\mathbb{N}^d$-graded, hence contained in the $\mathbb{N}^d$-graded maximal ideal $\mm$ of $R$. So each irreducible components
of $\Spec R$ contains $\mm$. We then see that $\Spec R$ is connected. So the lemma is true.
\end{proof}

\begin{lem}\label{tfrdomain}
 If the toric face ring $k[\Mcc]$ is a domain, then $\Sigma$ is the face poset of a cone. In particular, $k[\Mcc]$ is an affine monoid ring.
\end{lem}
\begin{proof}
Choose a system of generators $a_1,\ldots,a_n$ of $\Mcc$ where all the $a_i$ are non-zero. If $\Sigma$ has more than one maximal
cone then $a_1\cdots a_n=0$, which contradicts the condition that $k[\Mcc]$ is a domain.
\end{proof}

We deduce the first main result of this section.

\begin{thm}
\label{normalityisdegenerate}
Assume that $k[\Mcc]$ is normal. Then $\Sigma$ is the face poset of a cone and $k[\Mcc]$ is a normal affine monoid ring. In particular, $k[\Mcc]$ is
Cohen-Macaulay.
\end{thm}
\begin{proof}
Notice that toric face rings are always reduced, so we are in position to apply Lemma \ref{reducedgraded} and Lemma \ref{tfrdomain}.
 Hochster's theorem (\cite[Theorem 1]{Hochster}) says that normal affine monoid rings are Cohen-Macaulay, see
\cite[Theorem 6.3.5]{BH} for a proof using local cohomology. This concludes the proof of the theorem.
\end{proof}

Next, we consider seminormality of toric face rings. We have the second main result of this section as follows.

\begin{prop}
\label{seminormal-tfr}
If $k[\Mcc]$ is a seminormal ring then for every cone $C\in \Sigma$, the affine monoid ring $k[M_C]$ is seminormal. Conversely,
if for each cone $C\in \Sigma$, the monoid $M_C$ is seminormal, then the toric face ring $k[\Mcc]$ is also seminormal.
\end{prop}

First we have a simple remark.
\begin{lem}
Let $R\hookrightarrow S$ be an algebra retract of reduced rings. If $S$ is seminormal then so is $R$.
\end{lem}
\begin{proof}
If $x^2=y^3$ in $R$, then since $S$ is seminormal, there is some $z\in S$ such that $x=z^3, y=z^2$. Apply the retracting morphism $S\to R$,
we are done.
\end{proof}

Since for each $C\in \Sigma$, we have an algebra retract $k[M_C]\hookrightarrow k[\Mcc]$, so the first statement is true.
The second statement is a consequence of Proposition \ref{inverselimit} and Proposition \ref{inverselimitofseminormality}.

Another way to see that is by using the result in the book of Bruns and Gubeladze \cite[Exercise 8.13]{BG} which says that
if for each cone $C\in \Sigma$, the monoid $M_C$ is seminormal, then all finitely generated projective modules over $k[\Mcc]$ are
free. It is an easy exercise to show that the polynomial extensions of a toric face ring are again toric face rings. Moreover,
the property: that the monoid $M_C$ is seminormal for all $C\in\Sigma$, is stable under those polynomial extensions. Thus the Picard
groups of $k[\Mcc]$ and $k[\Mcc][X]$ (where $X$ is an indeterminate) are trivial. Apply Swan's theorem \cite[Theorem 1]{Swa},
which generalized Traverso's theorem \cite[Theorem 3.6]{Tra}, we have the conclusion that $k[\Mcc]$ is seminormal.

Given a monoidal complex $\Mcc$, we can define its {\em seminormalization} complex $\+{\Mcc}$ as follows. For each $C$, let $\+{M_C}$ be the
seminormalization of $M_C$ inside $\mathbb{Z}M_C$. Then $\+{\Mcc}$ is the collection of affine monoids $\+{M_C}$ with
$C\in \Sigma$. Denote $\+R= k[\+{\Mcc}]$.
\begin{thm}[\cite{BG}, Exercise 8.13]
$\+{\Mcc}$ is a seminormal monoidal complex supported on $\Sigma$ and the natural inclusion $R\hookrightarrow \+R$ is finite. Moreover, $R$ is
seminormal if and only if $\Mcc=\+{\Mcc}$.
\end{thm}
\begin{proof}
It is known that each $\+{M_C}$ is an affine monoid. We can check that $\+{\Mcc}$ is a monoidal complex supported on $\Sigma$.
That $\+{\Mcc}$ is seminormal follows from Proposition \ref{seminormal-tfr}.
Since for each $C\in \Sigma$, we have $M_C \subseteq \+M_C \subseteq \overline{M_C}$, the second and the last statement are also true.
\end{proof}

\section{Local cohomology of seminormal toric face rings}
\label{lc}

In this section, we generalize previous results of \cite{BBR}, \cite{BLR}, \cite{IR} concerning local cohomology of toric face
rings. We ask a question which amounts to a characterization of seminormal toric face rings via the vanishing of their local cohomology modules. We keep using
the notation of Section \ref{semi}. Note that $R$ is $\mathbb{Z}^d$-graded. Hence, all the local cohomology modules of $R$ are $\mathbb{Z}^d$-graded. For this reason, we will
restrict our attention to the $\mathbb{Z}^d$-graded components of $H^i_{\mm}(R), i\ge 0.$

First we recall basic constructions and facts about cell complex. Given a rational pointed fan $\Sigma \subseteq \mathbb{R}^d$, we associate a finite
regular cell complex $(X, \Gamma_{\Sigma})$ as follows. Let $X=\Sigma \cap \mathbb{S}^{d-1}$ where $\mathbb{S}^{d-1}$ is the $(d-1)$-dimensional unit sphere
in $\mathbb{R}^d$. Let $\Gamma_{\Sigma}=\{\relint (C)\cap \mathbb{S}^{d-1}: C\in \Sigma\}$.

For each $C\in \Sigma$, denote $e_C=\relint (C)\cap \mathbb{S}^{d-1}$. Each $e_C$ is an open cell. Denote $\Gamma^i_{\Sigma}=\{e\in
\Gamma_{\Sigma}: \bar{e} ~ \textnormal{homeomorphic to} ~ \mathbb{B}^i \}$, where $\mathbb{B}^i$ is the $i$-dimensional ball in $\mathbb{R}^i$. Then
$\cup_{j\le i}\Gamma^j_{\Sigma}$ is the $i$-skeleton of $\Gamma_{\Sigma}$. The dimension of $\Gamma_{\Sigma}$ is $\dim \Gamma_{\Sigma} =\max \{i:
\Gamma^i_{\Sigma} \neq \emptyset\}=\dim \Sigma -1.$ We say that $e_{C'}$ is a face of $e_C$ if $C'$ is a face of the cone $C$.

There is an incidence function $\delta(.,.)$ on pairs of cells $e_C, e_{C'}$ of $\Gamma_{\Sigma}$ with $e_C \in \Gamma^i_{\Sigma}$ and
$e_{C'} \in \Gamma^{i-1}_{\Sigma}$ for some $i\ge 0$. For each such pair of cells, $\delta(e_C, e_{C'}) \in \{0, \pm{1}\}$ and
$\delta(.,.)$ satisfies the following conditions:
\begin{enumerate}
\item $\delta(e_C, e_{C'})\neq 0$ if and only if $e_{C'}$ is a face of $e_C$;
\item $\delta(e_C,\emptyset)=1$ for each $0$-cell $e_C$;
\item if $e_{C'} \in \Gamma^{i-2}_{\Sigma}$ is a face of $e_C \in \Gamma^i_{\Sigma}$ then
\[
\delta(e_C, e_{C_1})\delta(e_{C_1},e_{C'})+\delta(e_C, e_{C_2})\delta(e_{C_2},e_{C'})=0
\]
where $e_{C_1}, e_{C_2}$ are the uniquely determined $(i-1)$-cells such that $e_{C'}$ is a face of $e_{C_i}$ and $e_{C_i}$ is face of $e_C$.
\end{enumerate}

We can now define the augmented oriented chain complex of $\Gamma_{\Sigma}$ as follows:
\[\mathcal{C}_{\pnt}(\Gamma_{\Sigma}): 0\to \mathcal{C}_{\dim \Gamma_{\Sigma}}(\Gamma_{\Sigma})\to \ldots \to \mathcal{C}_{0}(\Gamma_{\Sigma})\to
\mathcal{C}_{-1}(\Gamma_{\Sigma})\to 0
\]
where $\mathcal{C}_{i}(\Gamma_{\Sigma})=\bigoplus_{e_C\in \Gamma^i_{\Sigma}}\mathbb{Z}e_C$ for $i=0,\ldots, \dim \Gamma_{\Sigma}$, and
 $\mathcal{C}_{-1}(\Gamma_{\Sigma})=\mathbb{Z}$. The differential $\partial$ is defined on $\Gamma^i_{\Sigma}$ as follows:
\[
\partial(e_C)=\sum _{e_{C'}\in \Gamma^{i-1}_{\Sigma}}\delta(e_C,e_{C'})e_{C'}.
\]
Denote by $\widetilde{H}_i(\Gamma_{\Sigma})$ the $i$-th homology of $\mathcal{C}_{\pnt}(\Gamma_{\Sigma})$.

The local cohomology modules of toric face rings are computed by the following version of \v{C}ech comlex. For each cone $C$ of
$\Sigma$, denote by $R_C$ the localization of $R$ at the multiplicative closed set $\{X^a: a\in M_C\}$. Define the $R$-modules
$$
L^t(\Mcc)=\bigoplus _{C \in \Sigma, \ \dim C=t}R_C, ~ t=0,\ldots,\dim \Sigma,
$$
and define the differential $\partial :L^{t-1} \to L^t$ componentwise as follows: the map $\partial _{C,C'}:R_{C'} \to R_C$ is $\delta(e_C, e_{C'})
\textnormal{nat}$, where ``$\textnormal{nat}$" is the natural localization map.
The following theorem is a generalization of the computation of local cohomology of affine monoid rings \cite[Theorem 6.2.5]{BH}.
\begin{thm}[\cite{IR}, Theorem 4.2]
\label{Cech}
The complex $L^{\pnt}(\Mcc)$ defined above computes the local cohomology of an arbitrary $R$-module $G$. Hence for all $i\ge 0$, we have
$$
H^i_{\mm}(G)\cong H^i(L^{\pnt}(\Mcc) \otimes _R G).
$$
\end{thm}

We will also need the Mayer-Vietoris sequence for local cohomology of toric face rings. Recall that for a subfan $\Sigma '$ of $\Sigma$,
 we have the induced monoidal complex $\Mcc_{\Sigma'}$ and an induced toric face ring $R_{\Sigma'}=k[\Mcc_{\Sigma'}]$. (This is not
to be confused with the localization $R_C$ described above.)
There's a natural surjection  $R=R_{\Sigma} \to R_{\Sigma'}$ mapping every homogeneous elements outside $\Sigma'$ to zero, which preserves
the graded maximal ideals. Hence the local cohomology modules of $R_{\Sigma'}$ are the same when we consider it as an $R$-module.

Moreover, if $\Sigma = \Sigma_1\cup \Sigma_2$ for two subfans $\Sigma_1, \Sigma_2$ then we have a short exact sequence of $R$-modules:
$$
0\to R\to R_{\Sigma_1}\oplus R_{\Sigma_2}\to R_{\Sigma_1\cap\Sigma_2}\to 0.
$$
This gives rise to the following result, which might be called the Mayer-Vietoris sequence of local cohomology.

\begin{thm}[\cite{IR}, Proposition 4.3]
Let $\Mcc$ be a monoidal complex supported by a rational pointed fan $\Sigma$ in $\mathbb{R}^d$. Suppose that $\Sigma$ is the union of
two subfans, $\Sigma = \Sigma_1\cup \Sigma_2$. Then there is an exact sequence of $\mathbb{Z}^d$-graded $R$-modules
$$
\ldots\to H^{i-1}_{\mm}(R_{\Sigma_1\cap\Sigma_2})\to H^i_{\mm}(R)\to H^i_{\mm}(R_{\Sigma_1})\oplus
H^i_{\mm}(R_{\Sigma_2})\to H^{i}_{\mm}(R_{\Sigma_1\cap\Sigma_2})\to \ldots
$$
\end{thm}
Now we have the following theorem concerning the vanishing of local cohomology of seminormal toric face rings, which extends \cite[Theorem 4.3]{BLR} and \cite[Proposition 4.4]{IR}.
\begin{thm}
\label{vanishing}
Let $\Sigma$ be a rational pointed fan in $\mathbb{R}^d$ (where $d\ge 1$), $\Mcc$ be a seminormal monoidal complex supported on $\Sigma$ and $R=k[\Mcc]$.
 Assume that $H^i_{\mm}(R)_a \neq 0$ for some $a\in \mathbb{Z}^d$. Then $a\in -\overline{M_C}$ for a cone $C\in \Sigma$ of dimension $\le i$. In particular,
$$
H^i_{\mm}(R)_a = 0 ~ \textnormal{if} ~ a\notin -|\overline{\mathcal M}|=\cup _{C\in \Sigma}(-\overline{M_C}).
$$
\end{thm}

\begin{proof}
If $i=0$, since $R$ is reduced, we have $H^0_{\mm}(R)= 0$ and thus there's nothing to do. Assume that $i>0$ and
 $a\notin -\overline{M_D}$ for any cone $D\in \Sigma$ of dimension $\le i$. We will prove that $H^i_{\mm}(R)_a = 0.$

If $\dim \Sigma=0$ then $R\cong k$ and the claim is clearly true. Assume that $\dim \Sigma >0$. Let $C\in \Sigma$ be a cone of maximal dimension
$\dim C= \dim \Sigma$.

Let $\Sigma_1=\Sigma -{C}$ and $\Sigma_2= \Fc(C)$ be the face poset of $C$. We have the Mayer-Vietoris sequence:
$$
\ldots\to H^{i-1}_{\mm}(R_{\Sigma_1\cap\Sigma_2})_a\to H^i_{\mm}(R)_a\to H^i_{\mm}(R_{\Sigma_1})_a\oplus
H^i_{\mm}(R_{\Sigma_2})_a\to \ldots
$$
Note that $R_{\Sigma_2} \cong k[M_C]$. If $a\notin \mathbb{Z}M_C$, we have $H^i_{\mm}(R_{\Sigma_2})_a=0$, since
$H^i_{\mm}(k[M_C])$ is $\mathbb{Z}M_C$-graded. Assume that $a\in \mathbb{Z}M_C$. Since $M_C$ is seminormal and
$a \notin -\overline{M_D}$ for any face $D$ with dimension $\le i$ of $C$, Theorem 4.3 in \cite{BLR} implies that
$H^i_{\mm}(R_{\Sigma_2})_a=0$. Moreover, $a\notin -\overline{M_D}$ if $D$ is either a cone of dimension $\le i$ of $\Sigma_1$
or a cone of dimension $\le i-1$ of $|\Sigma_1\cap\Sigma_2|$. Thus by setting up an
 induction on $i$, another on the dimension and yet another induction on the number of cones of maximal dimension of a fan, we may assume that
$H^{i-1}_{\mm}(R_{\Sigma_1\cap\Sigma_2})_a=0$ and $H^i_{\mm}(R_{\Sigma_1})_a=0$. From the long exact sequence we see that $H^i_{\mm}(R)_a = 0$,
as claimed.
\end{proof}

Next we present a computation of local cohomology of seminormal toric face rings in combinatorial terms, that is, via homology of
certain cell complexes. We give an application of this computation, namely to deduce Brun, Bruns and R\"omer's generalized version of Hochster's formula
 for local cohomology of Stanley-Reisner rings.

\begin{defn}
\label{star}
 For each $a\in \mathbb{Z}^d$, denote by $\st _{\Sigma}(a)=\{D\in \Sigma: a \in \overline{M_D}\}$
 and $\Sigma(a)=\Sigma \setminus \st_{\Sigma}(a)$, which is a subfan of $\Sigma$. Denote by $\mathcal{C}_{\pnt}(\Gamma_{\st_{\Sigma}(a)})$ the
complex $\mathcal{C}_{\pnt}(\Gamma_{\Sigma})/\mathcal{C}_{\pnt}(\Gamma_{\Sigma(a)}).$ The complex
$\Hom _{\mathbb Z}(\mathcal{C}_{\pnt}(\Gamma_{\st_{\Sigma}(a)}),k)$ is denoted by $\mathcal{C}^{\pnt}(\Gamma_{\st_{\Sigma}(a)}).$ Denote
the corresponding homology and cohomology of the above complexes by $\widetilde{H}_i(\Gamma_{\st_{\Sigma}(a)})$ and
$\widetilde{H}^i(\Gamma_{\st_{\Sigma}(a)})$, respectively.
\end{defn}

In the following, as usual, given a complex $\mathcal{C}_{\pnt}$ the notation $\mathcal{C}_{\pnt}[m]$ denotes the complex
$\mathcal{C}_{\pnt}$ right-shifted by $m$ positions, so $\mathcal{C}_{i}[m]=\mathcal{C}_{i+m}.$ Theorem \ref{lc-combin} generalizes
Theorem 4.5 in \cite{IR}.
\begin{thm}
\label{lc-combin}
Let $\mathcal M$ be a monoidal complex supported on the rational pointed fan $\Sigma$ in $\mathbb{R}^d$. Let $a\in \mathbb{Z}^d$. Then
\[
L^{\pnt}(\mathcal M)_a\cong L^{\pnt}(\Mcc_{\Sigma(-a)})_a\oplus
\Hom _{\mathbb Z}(\mathcal{C}_{\pnt}(\Gamma_{\st _{\Sigma}(-a)})[-1],k)\otimes _{k}k(-a)
\]
as complexes of $\mathbb{Z}^d$-graded $k$-vector spaces. Hence for all $i\ge 0$ we have an isomorphism of graded $k$-vector spaces:
\[
H^i_{\mm}(R)_a \cong H^i_{\mm}(R_{\Sigma(-a)})_a \oplus \widetilde{H}^{i-1}(\Gamma_{\st _{\Sigma}(-a)})\otimes_{k}k(-a).
\]

If in addition, $\mathcal M$ is seminormal then for each $i\ge 0$,
\[
H^i_{\mm}(R)_a \cong \widetilde{H}^{i-1}(\Gamma_{\st _{\Sigma}(-a)})\otimes_{k}k(-a).
\]
\end{thm}

First we prove an auxiliary result.
\begin{lem}
\label{iso-cone}
We have the following:
\begin{enumerate}
 \item If $C\in \Sigma(-a)$ then $(R_C)_a \cong (k[{\mathcal M}_{\Sigma(-a)}]_C)_a$.

\item If $C'\subseteq C$ are cones of $\Sigma$ such that $C'\in \Sigma(-a), C\in \st_{\Sigma}(-a)$ then the natural map $(R_{C'})_a \to (R_C)_a$ is zero.
\end{enumerate}
\end{lem}
\begin{proof}
(i) Note that we have a short exact sequence of graded $R$-modules:
$$
0\to \qq _{\Sigma(-a)}\to R \to k[\Mcc_{\Sigma(-a)}]\to 0.
$$
After localizing, we are left with proving that $({\qq _{\Sigma(-a)}}_C)_a=0.$
Assume the contrary, so there exists $0\neq X^z/X^y\in {\qq _{\Sigma(-a)}}_C$ with $z\notin |\Mcc_{\Sigma(-a)}|$ and $y\in M_C$ such that $z-y=a$.

Since $X^z/X^y \neq 0$, there's a cone $D$ for which $M_D$ contains $z$ and $y$. Since $z\in M_D$ and $z\notin |\Mcc_{\Sigma(-a)}|$,
we must have $D\in \st_{\Sigma(-a)}$, and thus $-a\in \overline{M_D}$.

Since $y\in M_{C\cap D}$, replace $C$ by $C\cap D$ if necessary, we can assume that $C\subseteq D$. Now $-a, z$ are elements of $D$ such that
$-a +z=y$ and $y\in M_C$. But $C$ is a face of $D$, so we have $z\in C\cap M_D=M_C$. But then $z\in |\Mcc_{\Sigma(-a)}|$, contradiction.

(ii)  Clearly $-a\in C$. Since $-a\in \overline{M_C}=\Z M_C\cap C$, we can write $-a=u-w$ where $u,w\in M_C$. First note that
$u\notin C'$, otherwise $(-a)+w\in C'$, hence $-a\in C'$ and $u,w\in C'\cap M_C= M_{C'}$. Therefore $-a\in \Z M_{C'}\cap C'=\overline{M_{C'}}$, a contradiction.

Assume that $X^z/X^y \in R_{C'}$ with $z\in |\Mcc|,y\in M_{C'}$ and $z-y=a$. We will prove that $X^uX^z=0$ in $R$. Then since $X^u$ is a unit in $R_C$, we have $X^z/X^y=0$ in $(R_C)_a$.

It is enough to show that there is no cone $D\in \Sigma$ such that $z,u\in M_D$. Assume that there is such a cone. Since
$z-y=a$, we have $y+w=z+u\in D$. But $y,w\in C$ so both of them are in $C\cap D$. The same argument implies that $z,u\in C\cap D$. Now $(-a)+z=y$ and $-a,z\in C, y\in C'\subseteq C$, so $-a,z\in C'$. But then $z\in M_{C'}$, therefore $-a=y-z\in \Z M_{C'}\cap C'=\overline{M_{C'}},$ a contradiction. The proof of the lemma is completed.
\end{proof}

\begin{proof}[Proof of Theorem \ref{lc-combin}]
We notice that it is enough to prove the first isomorphism. The last statement follows easily from the first isomorphism and Theorem
\ref{vanishing} since in this case $\Mcc_{\Sigma(-a)}$ is seminormal and $a \notin -|\overline{\Mcc_{\Sigma(-a)}}|$. Thus
 $H^i_{\mm}(R_{\Sigma(-a)})_a=0$ for all $a$.

If $\st_{\Sigma}(-a)=\emptyset$ then the isomorphism is trivial. Thus in the following we assume that $\st_{\Sigma}(-a)\neq \emptyset$. We have
$$
L^i(\mathcal M)_a = \bigoplus _{C\in \Sigma(-a), ~ \dim C=i} (R_C)_a  \oplus \bigoplus _{D\in \st_{\Sigma}(-a), ~ \dim D=i} (R_D)_a.
$$
For each $D\in \st_{\Sigma}(-a)$, we have $-a \in M_D$. Thus $X^{-a}$ is a non-zero element of $(R_D)_a$, hence $(R_D)_a\cong k$. We get
$$
\bigoplus _{D\in \st_{\Sigma}(-a), ~ \dim D=i} (R_D)_a\cong
\Hom _{\mathbb Z}(\mathcal{C}_{i-1}(\Gamma_{\st _{\Sigma}(-a)}),k)\otimes _{k}k(-a),
$$
as $k$-vector spaces. Combine with Lemma \ref{iso-cone}, the first isomorphism is proved on the module level.

Next, we prove the isomorphism on the graded complex level and thus finish the proof of the theorem. Consider $C'\subseteq C$ with $C'\in
\Gamma^{i}_{\Sigma}, C\in \Gamma^{i+1}_{\Sigma}$. There are three cases to consider.

\smallskip

\noindent {\em Case 1:} If both $C'$ and $C$ belong to $\in \Sigma(-a)$. We can check that the following diagram, with two
vertical maps being isomorphisms, is commutative.
\begin{displaymath}
    \xymatrix{ (R_{C'})_a \ar[r] \ar[d] & (R_C)_a \ar[d] \\
               (k[{\mathcal M}_{\Sigma(-a)}]_{C'})_a  \ar[r]  & (k[{\mathcal M}_{\Sigma(-a)}]_C)_a}
\end{displaymath}

\smallskip

\noindent {\em Case 2:} If $C'$ belongs to $\st_{\Sigma}(-a)$, then so does $C$. It is easy to check that the following diagram is commutative.
\begin{displaymath}
    \xymatrix{ (R_{C'})_a \ar[r] \ar[d] & (R_C)_a \ar[d] \\
               \Hom _{\Z}(\mathcal{C}_{i-1}(\Gamma_{\st _{\Sigma}(-a)}),k)\otimes_{k}k(-a)  \ar[r]  & \Hom _{\Z}(\mathcal{C}_{i}(\Gamma_{\st _{\Sigma}(-a)}),k)\otimes_{k}k(-a)}
\end{displaymath}
\noindent {\em Case 3:} If $C'\in \Sigma(-a)$ but $C\in \st_{\Sigma}(-a)$. In the diagram
\begin{displaymath}
    \xymatrix{ (R_{C'})_a \ar[r] \ar[d] & (R_C)_a \ar[d] \\
               (k[{\mathcal M}_{\Sigma(-a)}]_{C'})_a  \ar[r]  & \Hom _{\Z}(\mathcal{C}_{i}(\Gamma_{\st _{\Sigma}(-a)}),k)\otimes_{k}k(-a)}
\end{displaymath}
the horizontal map below is zero. According to Lemma \ref{iso-cone} the above horizontal map is also zero. This concludes the proof of the theorem.
\end{proof}

\begin{prop}
\label{compare}
We always have
$$
H^i_{\mm}(\+R)=\bigoplus_{a\in -|\overline{\mathcal M}|}H^i_{\mm}(\+R)_a,
$$
and $H^i_{\mm}(\+R)$ is a $k$-direct summand of $\bigoplus_{a\in -|\overline{\mathcal M}|}H^i_{\mm}(R)_a.$
\end{prop}
\begin{proof}
Since $\+R$ is seminormal, Theorem \ref{vanishing} shows that $H^i_{\mm}(\+R)_a=0$ for $~ a\notin -|\overline{\mathcal M}|$.
Thus the first statement is clear. The second statement follows from Theorem \ref{lc-combin}. Indeed, for each $C\in \Sigma$, we have $M_C \subseteq \+{M_C}
\subseteq \overline{M_C}$ so the sets $\Sigma(a)$ and $\st_{\Sigma}(a)$ do not change when one passes from $\Mcc$ to $\+\Mcc$. Thus
\[
H^i_{\mm}(R)_a\cong H^i_{\mm}(R_{\Sigma(-a)})_a \oplus H^i_{\mm}(\+R)_a
\]
and this implies the desired conclusion.
\end{proof}
\begin{cor}
\label{monoidvanishing}
Let $M$ be a positive affine monoid and $a\in -\overline{M}$. Denote by $\Sigma$ the cone $\R_+M$. Then
\[
H^i_{\mm}(k[M_{\Sigma(-a)}])_a =0.
\]
\end{cor}
This follows from the proof of Proposition \ref{compare} and the fact that
$$
H^i_{\mm}(k[M])_a\cong H^i_{\mm}(\+{k[M]})_a,
$$
for $a\in -\overline{M}$, see \cite[Proposition 4.4]{BLR}.

We have a complete analog of \cite[Corollary 4.6]{BLR} with the same proof. It shows that for properties like Cohen-Macaulayness and Serre's
condition $(S_r)$, restriction to the class of seminormal toric face rings is reasonable.
\begin{cor}
Let $\Mcc$ be a monoidal complex supported on a rational pointed fan $\Sigma$ in $\R^d$, and $R=k[\Mcc]$. Then:
\begin{enumerate}
 \item If $\Mcc$ is Cohen-Macaulay over $k$, then so is $\+{\Mcc}$.
\item If $\depth R\ge r$ then $\depth \+R \ge r$.
\item If $R$ satisfies $(S_r)$ then $\+R$ satisfies $(S_r)$.
\end{enumerate}

\end{cor}

In the case of affine monoid rings, the following theorem was proved in \cite{BLR}.
\begin{thm}[\cite{BLR}, Theorem 4.7]
\label{cohomologycharacterization}
Let $M$ be a positive affine monoid. Then the following are equivalent:
\begin{enumerate}
 \item $M$ is seminormal;
\item $H^i_{\mm}(R)_a=0$ for all $i$ and all $a\in \Z M$ such that $a\notin -\overline{M}$.
\end{enumerate}

\end{thm}

We would like to extend the cohomological characterization of seminormality of Bruns, Li and R\"omer to toric face rings. However, we do not have
an answer for the following question.

\begin{quest}
\label{lctosemi}
Is it true that for a toric face ring $k[\Mcc]$ the following statements are equivalent?
\begin{enumerate}
 \item $\mathcal M$ is seminormal;
\item $H^i_{\mm}(R)_a=0$ for all $i$ and all $a\in \Z^d$ such that $a\notin -|\overline{\mathcal M}|$.
\end{enumerate}
\end{quest}

\begin{rem}
 The method to prove Theorem \ref{cohomologycharacterization} given in \cite{BLR} cannot be directly generalized to deal with toric face rings. Analyzing this
proof, we observe that it depends crucially on Corollary \ref{monoidvanishing}. However, in contrast to the situation of affine monoid rings,
the next example shows that for some $2$-dimensional toric face ring $k[\Mcc]$ and some $a\in -|\overline{\mathcal M}|$,
it can happen that $H^2_{\mm}(R_{\Sigma(-a)})_a \neq 0$.
\end{rem}

\begin{ex}
\label{toricfaceringnonvanishing}
Let $\Sigma$ be the fan in $\R ^2$ consists of two maximal cones $C$ with generators $x=(3,0), y=(3,1), z=(3,3)$ and $C'$ with two
 generators $z$ and $t=(0,1)$. Define $\Mcc$ to be the monoidal complex supported on $\Sigma$ with two maximal monoids: $M$ generated by
$x,y, z$, $M'$ generated by $z, t$. The toric face ring $k[\Mcc]$ is $k[x,y,z,t]/(x^2z-y^3,xt,yt)$.

Let $D$ be the ray spanned by $t$, let $a=-t=(0,-1)$.

\bigskip
\bigskip

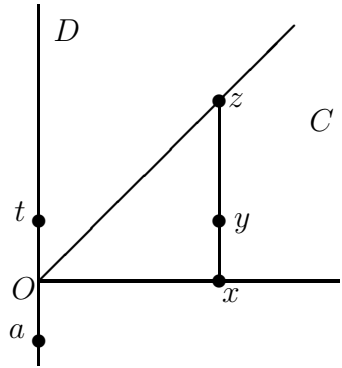
\setlength{\unitlength}{4cm}
\begin{picture}(1,1)
\thicklines
\put(1.5,-0.28){\line(0,1){1.2}}
\put(1.5,0){\line(1,0){1}}
\put(1.5,0){\line(1,1){0.85}}

\put(1.41,-0.06){$O$}
\put(1.4,-0.19){$a$}
\put(1.5,-0.2){\circle*{0.04}}
\put(2.11,-0.07){$x$}
\put(2.1,0){\circle*{0.04}}
\put(2.15,0.18){$y$}
\put(2.1,0.2){\circle*{0.04}}
\put(2.13,0.58){$z$}
\put(2.1,0.6){\circle*{0.04}}
\put(1.42,0.20){$t$}
\put(1.5,0.2){\circle*{0.04}}
\put(1.55,0.8){$D$}
\put(2.4,0.5){$C$}

\put(2.1,0.2){\line(0,2){.4}}
\put(2.1,0){\line(0,1){.2}}
\end{picture}

\bigskip
\bigskip

\bigskip
\bigskip

Note that $-a \in \overline{M_D}=M_D$ and $\Sigma(-a)= \Fc(C)$, the face poset of $C$. Thus
\[
H^2_{\mm}(R_{\Sigma(-a)})_a = H^2_{\mm}(k[M])_a \cong k,
\]
the last equality can be checked by computation on the \v{C}ech of the monoid ring $A=k[M]$,
\[
0 \to A \to A_x \oplus A_z \to A_C \to 0.
\]
Indeed, $x/y$ is a non-zero element of $(A _C)_a$ while $(A_x)_a = (A_z)_a=0$.
\end{ex}

Theorem 5.5 of \cite{BBR} delivered a formula for the graded local cohomology of toric face rings by the machinery of cohomology
 of sheaves on partially ordered set. This formula, which generalizes Hochster's formula for local cohomology of Stanley-Reisner rings, was reproved
in \cite[Corollary 4.7]{IR} by the more compact language of toric face rings. In the Corollary \ref{BBR-lc}, we derive this formula from Theorem
 \ref{lc-combin}.

\begin{defn}
For each cone $C$ of $\Sigma$, denote by $\st _{\Sigma}(C)=\{D\in \Sigma: C \subseteq D\}$. We also denote by $\Sigma(C)$
the set $\Sigma \setminus \st_{\Sigma}(C)$, which is a subfan of $\Sigma$. In the same way as in Definition \ref{star}, we define the complexes
$\mathcal{C}_{\pnt}(\Gamma_{\st_{\Sigma}(C)})$ and $\mathcal{C}^{\pnt}(\Gamma_{\st_{\Sigma}(C)})$, and the homology and cohomology groups
$\widetilde{H}_i(\Gamma_{\st_{\Sigma}(C)})$ and $\widetilde{H}^i(\Gamma_{\st_{\Sigma}(C)})$.
\end{defn}

The set $\st_{\Sigma}(C)$ is partially ordered by inclusion. Denote by $\Delta(\st_{\Sigma}(C))$ the order complex of
$\st_{\Sigma}(C)\setminus \{C\}$, which is the simplicial complex whose faces are the chains of $\st_{\Sigma}(C)\setminus \{C\}$. Denote by
$\widetilde{H}_i(\Delta(\st_{\Sigma}(C))), \widetilde{H}^i(\Delta(\st_{\Sigma}(C)))$ the simplicial homology and cohomology of $\Delta(\st_{\Sigma}(C))$ with
 coefficients in $k$.
\begin{lem}[\cite{IR}, Lemma 4.6]
\label{simplicialhom}
With the above notation we have for each $i\in \Z$:
$$
\widetilde{H}_i(\Gamma_{\st_{\Sigma}(C)})\cong \widetilde{H}_{i-\dim C}(\Delta(\st_{\Sigma}(C))),
\widetilde{H}^i(\Gamma_{\st_{\Sigma}(C)})\cong \widetilde{H}^{i-\dim C}(\Delta(\st_{\Sigma}(C)))
$$
\end{lem}

\begin{cor}[Brun, Bruns, R\"omer \cite{BBR}]
\label{BBR-lc}
Let $\Sigma$ be a rational pointed fan in $\R^d$, and $R=k[\Sigma]$. Then for all $i\ge 0$, there are isomorphisms of finely
 graded $k$-modules
\begin{align}
H^i_{\mm}(R)  &\cong \bigoplus _{C\in \Sigma}\bigoplus_{a \in -\relint C}\widetilde{H}^{i-1}(\Gamma_{\st _{\Sigma}(C)})\otimes_{k}k(-a) \nonumber\\
              &\cong \bigoplus _{C\in \Sigma}\bigoplus_{a \in -\relint C}\widetilde{H}^{i-\dim C-1}(\Delta(\st_{\Sigma}(C)))\otimes_{k}k(-a). \nonumber
\end{align}

\end{cor}
\begin{proof}
Note that for every $C\in \Sigma$ we have $M_C=C\cap \Z^d$ is a normal monoid. Thus $R$ is seminormal and $|\overline{\mathcal M}|=|\Sigma|$.
We also have $\st_{\Sigma}(a)=\{D \in \Sigma: a\in D\}$. Therefore $\st_{\Sigma}(a)=\st_{\Sigma}(C)$ if $a\in \relint C$.

From Theorem \ref{vanishing}, we can restrict our attention to graded pieces $H^i_{\mm}(R)_a$ where $a\in -|\Sigma|$. For each $a\in -|\Sigma|$,
there is a unique $C$ such that $-a\in \relint C$. So apply Theorem \ref{lc-combin} and Lemma \ref{simplicialhom} we have:
\begin{align}
H^i_{\mm}(R) & \cong \bigoplus _{a\in -|\Sigma|}\widetilde{H}^{i-1}(\Gamma_{\st _{\Sigma}(-a)})\otimes_{k}k(-a) \nonumber\\
             & \cong \bigoplus _{C\in \Sigma}\bigoplus_{a \in -\relint C}\widetilde{H}^{i-1}(\Gamma_{\st _{\Sigma}(C)})\otimes_{k}k(-a) \nonumber\\
             & \cong \bigoplus _{C\in \Sigma}\bigoplus_{a \in -\relint C}\widetilde{H}^{i-\dim C-1}(\Delta(\st_{\Sigma}(C)))\otimes_{k}k(-a). \nonumber
\end{align}
This concludes the proof of the corollary.

\end{proof}


\section{Depth and the Cohen-Macaulay property}
\label{depth}
In this section, we present an alternative description for the depth of a seminormal toric face ring. The main tool is Theorem \ref{lc-combin}.
\begin{defn}
For each $i=0,1,\ldots, \dim R$, denote by $\Sigma^{(i)}$ the set $\{C\in \Sigma: \dim C\le i\}$. We call $\Sigma^{(i)}$, which is a subfan of $\Sigma$,
 the {\em $i$-skeleton} of $\Sigma$. Denote by $\Mcc ^{(i)}$ the restriction $\Mcc _{\Sigma^{(i)}}$, which is called the
{\em $i$-skeleton} of $\Mcc$.
\end{defn}
Define the number
$$
m_k(\Mcc)=\max \{i\le \dim R : \Mcc^{(t)} \ \text{is Cohen-Macaulay over $k$ for all} \ 0\le t\le i\}.
$$
This makes sense because $k[\Mcc ^{(0)}]=k$ is Cohen-Macaulay. Actually if $\Sigma$ is not trivial, i.e.~ not only the origin, then $m_k(\Mcc) \ge 1$.
 Indeed, if $\Sigma$ is not trivial, it is easy to see that $k[\Mcc ^{(1)}]$ is a Stanley-Reisner ring of dimension $1$, which is well-known to be
 Cohen-Macaulay \cite[Exercise 5.1.26]{BH}.
We can now prove a rank-selection theorem, see also (in time order) Munkres \cite[Theorem 3.1]{Mun}, D. Smith \cite[Theorem 4.8]{Smi}, Hibi
 \cite[Corollary 2.6]{Hibi}, Duval \cite[Corollary 6.5]{Duva} and Brun and R\"omer \cite[Example 5.8]{BR2}.

\begin{thm}
If $\Mcc$ is seminormal then
\[
\depth R = m_k(\Mcc).
\]
In particular, if $\Mcc$ is seminormal and Cohen-Macaulay over $k$, then so are all of its skeletons.
\end{thm}
\label{rankselect}
\begin{proof}
Denote $s=m_k(\Mcc)$. Apply Theorem \ref{lc-combin}, for all $i\ge 0, a\in \Z^d$, we have
\[
H^i_{\mm}(R)_a \cong \widetilde{H}^{i-1}(\Gamma_{\st _{\Sigma}(-a)})\otimes_{k}k(-a).
\]
Note that by construction, the right-hand side involves only cones in $\Sigma$ of dimension $\le i+1$. Thus for $i\le s-1$ and
 $a\in \Z^d$,
$$
H^i_{\mm}(R)_a \cong H^i_{\mm}(k[\Mcc ^{(s)}])_a.
$$
This implies that $H^i_{\mm}(R)=H^i_{\mm}(k[\Mcc ^{(s)}])=0$, since $\Mcc ^{(s)}$ is Cohen-Macaulay of dimension $s$. Thus $\depth R\ge s$.

The same argument shows that $H^i_{\mm}(R)=H^i_{\mm}(k[\Mcc ^{(s+1)}])=0$ for $i\le s-1$. Since $\Mcc ^{(s+1)}$ has dimension $s+1$ and is not Cohen-Macaulay,
 we have $0 \neq H^s_{\mm}(k[\Mcc ^{(s+1)}])=H^s_{\mm}(R)$. We conclude that $\depth R=s$.
\end{proof}

Let $M$ be a seminormal affine monoid and define the number $c_k(M)$ to be the maximal number $t$
such that  $k[M\cap F]$ is Cohen-Macaulay for all faces $F$ of $\R _+M$ with dimension $\le t$.
According to Theorem 5.3 of \cite{BLR}, we have $\depth k[M]\ge c_k(M)$. Thus we have the following corollary.
\begin{cor}
\label{monoidrankcompare}
Let $M$ be a seminormal affine monoid. Then
\[
m_k(M)\ge c_k(M).
\]
\end{cor}

We can give another proof for Corollary \ref{monoidrankcompare} in the language of toric face rings. Indeed, if $c_k(M)=\dim k[M]$,
there's nothing to do. Otherwise, let $t=c_k(M)$. For every $i\le t$, the $i$-skeleton $M^{(i)}$ of the face poset of $\R_+M$ is
clearly pure shellable, see \cite{IR}, Section 3 for more discussion of shellability of toric face rings. Moreover, the monoid
rings on the faces of $M^{(i)}$ are Cohen-Macaulay. Thus apply Theorem 3.2 of \cite{IR}, we can conclude that $k[M^{(i)}]$ is
Cohen-Macaulay. This shows that $t\le m_k(M)$.


\section{Frobenius of toric face rings in positive characteristic}
\label{Frobenius}
\thispagestyle{empty}

In the following, assume that $k$ has positive characteristic $p$. Then we have the Frobenius endomorphism
of $R=k[\Mcc]$:\
\begin{center}
$F: R \rightarrow R, ~ t\mapsto t^p.$
\end{center}
$R$ is called {\em F-finite} if $R$ is a finite $F(R)$-module. We say that $R$ is {\em F-injective} if the induced maps of Frobenius on the local
 cohomology modules $H^i_{\mm}(R)$ are injective. We call $R$ to be {\em F-pure} if $F(R)$ is a pure subring of $R$, and {\em F-split} if $F(R)$
is a direct summand of $R$ as $F(R)$-module.

A finitely generated algebra over a perfect field of characteristic $p>0$ is $F$-finite. It is known that
$$
\text{$F$-split} \Rightarrow \text{$F$-pure} \Rightarrow \text{$F$-injective},
$$
and $F$-pure together with $F$-finite implies $F$-split, see Fedder \cite{Fed} for more information on these notions.

Hochster and Roberts \cite[Proposition 5.38]{HR} proved that Stanley-Reisner rings over $k$ are $F$-pure for all $\chara k>0$.
 Hochster and Roberts \cite[Theorem 5.33]{HR} also proved that positive affine seminormal monoid rings over $k$ are $F$-pure,
if $\chara k >0$ is different from a finite set of prime numbers. Bruns, Li and R\"omer describe these primes exactly.

\begin{thm}[\cite{BLR}, Proposition 6.2]
\label{F-injectivemonoid}
Let $M\subseteq \Z^d$ be a positive seminormal affine monoid and $k$ be field of characteristic $p>0$. Then the following are equivalent:
\begin{enumerate}
\item The prime ideal $(p)$ is not associated to the $\Z$-module $(\Z M\cap \R C)/\Z (M\cap C)$ for any face $C$ of $\R _+M$;
\item R is $F$-split;
\item R is $F$-pure;
\item R is $F$-injective.
\end{enumerate}
\end{thm}

T. Yasuda \cite[Proposition 5.3]{Yas} proved that for affine monoid rings, $F$-purity (even $F$-splitting) is
equivalent to {\em weak normality}. We would like to thank Karl Schwede for pointing to us this result of Yasuda. We will see
that similar results are true for face rings of seminormal monoidal complexes supported on a rational pointed fan.

\begin{rem}
\label{Finjectiveglobal}
If $\chara k=p>0$ and $k[\Mcc]$ is $F$-injective, then $\Mcc$ is seminormal. Indeed, we can pass to the perfect closure of $k$ since the
$F$-injectivity of $k[\Mcc]$ is not affected by a flat base change. We know that $k[\Mcc]$ is reduced and $F$-finite. Now apply a theorem of Schwede
\cite[Theorem 4.7]{Sch}, which says that if a reduced $F$-finite ring with a dualizing complex is $F$-injective, then it is seminormal
(actually even {\em weakly normal}). Thus $\Mcc$ is seminormal.
\end{rem}
Now we characterize the $F$-split and $F$-pure properties of seminormal toric face rings.
\newpage
\begin{thm}
\label{charac}
Let $\chara k = p >0$. Let $\Sigma$ be a rational pointed fan in $\R^d$, and $\Mcc$ a seminormal monoidal complex supported on $\Sigma$.
The following statements are equivalent:
\begin{enumerate}
\item For each maximal cone $C\in \Sigma$, the monoid ring $k[M_C]$ is $F$-pure (equivalently, $F$-injective);
\item For each maximal cone $C\in \Sigma$ and each face $D\subseteq C$, the prime ideal $(p)$ is not associated to the
$\Z$-module $(\Z M_C\cap \R M_D)/\Z M_D$;
\item $k[\Mcc]$ is $F$-split;
\item $k[\Mcc]$ is $F$-pure.
\end{enumerate}
\end{thm}

\begin{proof}

(iii)$\Rightarrow$(iv) is trivial. (iv)$\Rightarrow$(i) follows from a simple fact: if $R'\hookrightarrow R$ is an algebra retract of rings containing
a field of characteristic $p>0$ and $R$ is $F$-pure then $R'$ is $F$-pure. Note that (i)$\Rightarrow$(ii) follows Theorem \ref{F-injectivemonoid}.
 We prove that (ii)$\Rightarrow$(iii).

(ii)$\Rightarrow$(iii). We will prove that $k[p\Mcc]$ is a $k[p\Mcc]$-direct summand of $k[\Mcc]$. Then because $k^p[p\Mcc]$ is
a $k^p[p\Mcc]$-direct summand of $k[p\Mcc]$, we have $k^p[p\Mcc]$ is a $k^p[p\Mcc]$-direct summand of $k[\Mcc]$.

Define an equivalence relation $\sim$ on $|\Mcc|$ as follows: $a\sim b$ if and only if there is a finite sequence $a_1,\ldots, a_n$ such that
for $i=0,\ldots, n$ we have $a_i, a_{i+1}$ belongs to $M_C$ and $a_i-a_{i+1}\in p\Z M_C$ for some $C\in \Sigma$ (where $a_0=a, a_{n+1}=b$).
It is easy to check that the decomposition into equivalence classes of $|\Mcc|$ gives rise to $k[p\Mcc]$-module decomposition of $k[\Mcc]$ into
direct summands. Using (ii), we claim that $M_D\cap p\Z M_D =pM_D$ for every $D\in \Sigma$. From this, we will prove that the equivalence class of
 $0$ is exactly $p|\Mcc|$. This implies that $k[p\Mcc]$ is a direct summand in this $k[p\Mcc]$-module decomposition of
$k[\Mcc]$, thus finishes the proof.

Now assume that $z\in \Z M_D$ and $pz\in M_D$ with $D\in \Sigma$. Choose a maximal cone $C$ containing $D$ and a face $B$ of $D$ such that $z\in \relint B$.
Of course, $pz\in M_D\cap B\subseteq M_B$, $z\in \Z M_C$. Since $(p)$ is not associated to $(\Z M_C\cap \R M_B)/\Z M_B$ we have $z\in \Z M_B$.
Now $z\in \relint B\cap \Z M_B \subseteq M_B$, as $M_B$ is seminormal. So $z\in M_D$, as claimed.

Assume that there are elements $a_1,\ldots, a_n$ such that $a_i, a_{i+1}$ belongs to $M_C$ and $a_i-a_{i+1}\in p\Z M_C$ for some $C\in \Sigma$
(where $a_0=0$). We use induction on $n$ to show that  $a_n\in p|\Mcc|$.

If $n=1$, there is nothing to do. Assume that $n\ge 2$. We have $a_{n-1} \in pM_C$ for $C\in \Sigma$. For some $D\in \Sigma$, we have $a_{n-1}, a_n
\in M_D$ and $a_n-a_{n-1} \in p\Z M_D$. Of course $a_{n-1}\in pM_D$. Thus $a_n\in M_D\cap p\Z M_D =pM_D$, as desired.
\end{proof}

\begin{cor}[\cite{HR}, Proposition 5.38]
If $k$ is a field of characteristic $p>0$ and $I$ is a squarefree monomial ideal in $k[X_1,\ldots,X_n]$ (where $n\ge 1$), then $k[X_1,\ldots,X_n]/I$ is $F$-pure.
\end{cor}
\begin{proof}
Assume that $k[\Delta]$ is the Stanley-Reisner ring corresponding to $I$. Note that for each cone $C$ of the geometric realization of $\Delta$ in
 $\R^n$ and each face $D$ of $C$, we have $\Z M_C\cap \R M_D=\Z M_D.$ So $k[\Delta]$ is $F$-pure for every $p>0$, by Theorem
\ref{charac} (ii).
\end{proof}

\begin{rem}
 Note that from Remark \ref{Finjectiveglobal} and Theorem \ref{charac}, with given $\Mcc$, for all but finitely many values of $\chara k >0$, we have that if
$k[\Mcc]$ is $F$-injective then $k[M_C]$ is $F$-injective (even $F$-split) for every $C\in \Sigma$. The next example shows that we cannot expect more than
that.
\end{rem}
\begin{ex}
\label{Finjectivenonhereditary}

Let $k=\Z/2\Z$. We construct a monoidal complex $\Mcc$ supported on a two-dimensional fan $\Sigma$ such that $k[\Mcc]$ is
$F$-injective but $k[M_C]$ is not $F$-injective for some maximal cone $C$ of $\Sigma$.

Let $\Sigma$ be the fan in $\R ^2$ consisting of two maximal cones $C$ with generators $x=(1,0), y=(0,2), t=(1,1)$ and $C'$ with two generators $y$ and $z=(-2,2)$. Define $\Mcc$ to be the monoidal complex supported on $\Sigma$ with two maximal monoids: $M$ generated by
$x,y, t$, $M'$ generated by $y, z$. The toric face ring $k[\Mcc]$ is $k[x,y,z,t]/(x^2y-t^2,xz,tz)$.

Let $D$ be the ray spanned by $y$, let $a=(0,-1)$. Note that $\Mcc$ is seminormal because $M'$ is normal
and $M$ is seminormal.

\bigskip
\bigskip

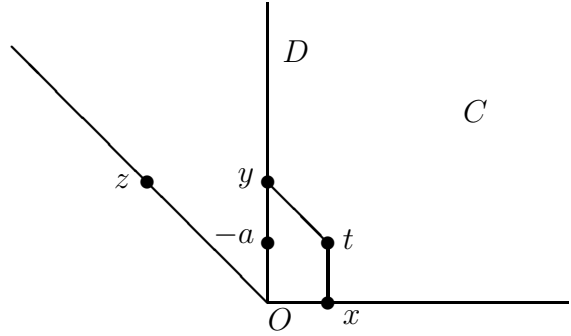
\setlength{\unitlength}{4cm}
\begin{picture}(1,1)
\thicklines
\put(1.5,0){\line(0,1){1}}
\put(1.5,0){\line(1,0){1}}
\put(1.5,0){\line(-1,1){0.85}}

\put(1.5,-0.09){$O$}
\put(1.4,0.40){$y$}
\put(1.5,0.4){\circle*{0.04}}
\put(1.75,-0.07){$x$}
\put(1.7,0){\circle*{0.04}}
\put(1.75,0.18){$t$}
\put(1.7,0.2){\circle*{0.04}}
\put(0.99,0.38){$z$}
\put(1.1,0.4){\circle*{0.04}}
\put(1.32,0.20){$-a$}
\put(1.5,0.2){\circle*{0.04}}
\put(1.55,0.8){$D$}
\put(2.15,0.6){$C$}

\put(1.5,0.4){\line(1,-1){.2}}
\put(1.7,0){\line(0,1){.2}}
\end{picture}

\bigskip
\bigskip

The last fact can be seen by checking the equality
\[
M=\bigcup_{F ~ \textnormal{is face of} ~ \R_+M}\Z (M\cap F) \cap \relint F.
\]
In detail, we have
\[ \Z (M\cap F) \cap \relint F= \begin{cases}  \{(m,n): m\ge 1, n\ge 1\} & \text{if $F=\R_+M$};\\
                                               \{(0,2n): n\ge 1\} & \text{if $F=D$};\\
                                               \{(m,0): m\ge 1\} & \text{if $F=Ox$};\\
                                               \{0\} & \text{if $F=\{0\}$.}
                                \end{cases}
\]

Thus
$$
\+{M}=\bigcup_{F ~ \textnormal{is face of} ~ \R_+M}\Z (M\cap F) \cap \relint F \subseteq M.
$$
and hence $M=\+{M}.$ It is not hard to see that $\overline{M} \setminus M= \{-a, -3a, -5a,\ldots\}$.

It is easy to check that $-a=t-x \in \Z M \cap \R D, -2a \in \Z M_D$ but $-a\notin \Z M_D$. Thus $k[M]$ is not $F$-injective. On the other hand,
we can check that $k[\Mcc]$ is $F$-injective.

Indeed, firstly, $R=k[\Mcc]$ is a 2-dimensional Cohen-Macaulay ring. This is done by a typical ``shellability'' argument. We have
an exact sequence of $R$-modules
\[
0\to R\to k[M]\oplus k[M'] \to k[M_D]\to 0.
\]
The middle rings are Cohen-Macaulay of dimension $2$ (see \cite[Corollary 5.4]{BLR}). Using the proof of \cite[Corollary 4.11]{BLR}, we can even
prove that $k[M]$ is Gorenstein. Indeed, the multigraded support of the $k$-dual of
$H^2_{\mm}(k[M])$ is
$$
N=\overline{M}\setminus \bigcup_{F ~ \textnormal{is facet of} ~ \R_+M}\Z (M\cap F)=(0,1)+M
$$
thus $k[M]$ is Gorenstein.

The last ring is Cohen-Macaulay of dimension $1$. Note that all of these rings have admissible $\N$-grading. Thus $R$ is Cohen-Macaulay, so
$H^i_{\mm}(R)=0, i=0,1.$

We can restrict our attention to those $b$ such that $H^2_{\mm}(R)_b \neq 0$. Apply Theorem \ref{vanishing}, this happens only when $-b\in
\overline{M}\cup \overline{M'} = \overline{M}\cup M'.$

Now apply Theorem \ref{lc-combin}. If $-b\in M\cup M'$, it is easy to see that $\st _{\Sigma}(-b)=\st _{\Sigma}(-2b)$, thus
$H^2_{\mm}(R)_b \cong H^2_{\mm}(R)_{2b}$ as $k$-vector spaces. We also have an isomorphism via Frobenius, by examining action of Frobenius
 on the Mayer-Vietoris sequence and applying the $5$-Lemma.

\begin{displaymath}
\xymatrix{ 0 \ar[r] & H^1_{\mm}(k[M_D])_b \ar[r] \ar[d]^{F} &  H^2_{\mm}(R)_b \ar[r] \ar[d]^{F} & H^2_{\mm}(k[M])_b  \oplus H^2_{\mm}(k[M'])_b \ar[r] \ar[d]^{F} & 0\\
           0 \ar[r] & H^1_{\mm}(k[M_D])_{2b} \ar[r]  &  H^2_{\mm}(R)_{2b} \ar[r] & H^2_{\mm}(k[M])_{2b}  \oplus H^2_{\mm}(k[M'])_{2b} \ar[r]  &  0}
\end{displaymath}

If $-b\in \overline{M}\setminus M$ then $b=a, 3a, 5a,\ldots$, in this case we have $\st _{\Sigma}(-b)=\{M\}$, $\st _{\Sigma}(-2b)=\{D, M, M'\}$,
and again by Theorem \ref{lc-combin}, we have $H^2_{\mm}(R)_b \cong k\cong H^2_{\mm}(R)_{2b}$. Examining action of Frobenius
 on the last part of the \v{C}ech complex shows that Frobenius induces an isomorphism $H^2_{\mm}(R)_b \cong H^2_{\mm}(R)_{2b}$.
\begin{displaymath}
    \xymatrix{ R_x \oplus R_y \oplus R_z \ar[r] \ar[d]^{F} &  R_C \oplus R_{C'} \ar[r] \ar[d]^{F} &  0\\
               R_x \oplus R_y \oplus R_z \ar[r]        &  R_C \oplus R_{C'} \ar[r] &  0}
\end{displaymath}

Thus $k[\Mcc]$ is $F$-injective.
\end{ex}

\begin{rem}
We observe that Example \ref{Finjectivenonhereditary} also gives a {\em negative} answer to the following question:

Let $R \hookrightarrow S$ is an algebra retract of finitely generated algebras over a field $k$ with $\chara k=p >0$. If $S$ is $F$-injective
and $R$ is Gorenstein, is it true that $R$ is $F$-injective?
\end{rem}

Finally, we prove that $R=k[\Mcc]$ is weakly $F$-regular if and only if $R$ is a normal affine monoid ring. Recall that if
$\chara k=p>0$, then $R$ is {\em weakly $F$-regular} if every ideal of $R$ is tightly closed. $R$ is {\em $F$-regular} if
the localization of $R$ at any multiplicative subset is weakly $F$-regular, see Hochster and Huneke \cite{HH} for more details.
Hochster and Huneke \cite[Corollary 5.11]{HH} proved that a weakly $F$-regular ring must be normal. Thus apply Theorem
\ref{normalityisdegenerate}, we have the following.
\begin{prop}
If $\chara k=p>0$ and $k[\Mcc]$ is weakly $F$-regular, then $\Sigma$ is the face poset of a cone and $k[\Mcc]$ is a normal affine
monoid ring.
\end{prop}

Note that normal affine monoid rings are always $F$-regular, because they are direct summands of Laurent polynomial
rings, see \cite[Exercise 6.1.10]{BH}.

\section*{Acknowledgment}

The author is grateful to Tim R\"omer for his generous suggestions and inspiring comments on the topic and the writing
of this paper. We wish to thank Matteo Varbaro for introducing to us the results of Hochster and Roberts \cite{HR}, and
suggesting the problem of characterizing $F$-purity of toric face rings. We are grateful to Neil Epstein for carefully mentoring
many interesting aspects of seminormality and Frobenius morphism and for his thorough literature guidance. He also read and
corrected to us several typos in an earlier draft.

We want to thank Alberto Fernandez Boix for correcting to us several typos and mentioning that Corollary \ref{BBR-lc} first
appeared in \cite{BBR}. Karl Schwede kindly informed us the result of Yasuda \cite[Proposition 5.3]{Yas} which is relevant in
our context.
%
%
%
%


\end{document}